\documentclass[11pt]{article}

\usepackage{graphicx}
\usepackage{amssymb,epsfig}
\usepackage{amsmath}
\usepackage[francais,english]{babel}
\usepackage[latin1]{inputenc}
\usepackage[T1]{fontenc}

\def\div{{\rm div \;}}

\def\N{\mathbb{N}}
\def\R{\mathbb{R}}

\def\E{\mathbb{E}} 
\def\P{\mathbb{P}} 

\newtheorem{proposition}{Proposition}

\newtheorem{lemma}{Lemma}
\newtheorem{remark}{Remark}

\def\acknowledgement{\medskip {\bf {Acknowledgements}: }}

\title{A mathematical formalization\\ of the parallel replica dynamics}
\author{C. Le Bris$^1$, T. Leli\`evre$^1$,  M. Luskin$^2$, D. Perez$^3$\\
{\footnotesize 1- CERMICS, \'Ecole des Ponts ParisTech,}\\ 
{\footnotesize 6 \& 8, avenue Blaise Pascal,}\\
{\footnotesize  77455 Marne-La-Vall\'ee, FRANCE}\\ 
{\footnotesize and}\\
{\footnotesize  INRIA Rocquencourt, MICMAC project-team,}\\
{\footnotesize  Domaine de Voluceau, B.P. 105,}\\
{\footnotesize  78153 Le Chesnay Cedex, FRANCE.}\\
{\footnotesize\tt \{lebris,lelievre\}@cermics.enpc.fr}
\\
{\footnotesize 2- School of Mathematics, University of Minnesota,}\\
{\footnotesize  206 Church Street SE,
Minneapolis, MN 55455, USA}\\
{\footnotesize\tt luskin@math.umn.edu}
\\
{\footnotesize 3- Theoretical Division T-1, Los Alamos National Laboratory, Los Alamos, NM, 87544, USA}\\
{\footnotesize\tt danny\_perez@lanl.gov}}

\begin{document}

\selectlanguage{english}{

\maketitle

\section{Motivation and context}

The purpose of this article is to lay the mathematical foundations of a
well known numerical approach in computational statistical physics, namely the {\emph{parallel replica dynamics}, introduced by A.F.~Voter in~\cite{voter-98}
 and improved and popularized
 in the context of Molecular Dynamics simulations in~\cite{voter-montalenti-germann-02,kum-dickson-stuart-uberuaga-voter-04,uberuaga-stuart-voter-07,uberuaga-hoagland-voter-valone-07},
 for example. The aim of the approach is to efficiently generate a
 coarse-grained evolution (in terms of state-to-state dynamics) of a
 given stochastic process. The approach formally consists in
 concurrently considering several realizations of the stochastic
 process, and tracking among the realizations that which, the soonest, undergoes
 an 
 important transition. Using specific properties of the dynamics
 generated, a computational speed-up is obtained. In the best cases,
 this speed-up approaches the number of realizations considered.  

By drawing connections with the theory of Markov processes and, in
particular, 
exploiting the notion of \emph{quasi-stationary distribution}, we provide a mathematical setting
appropriate for assessing theoretically the performance of the approach, and possibly improving it.

\subsection{Description of the parallel replica dynamics}}

Consider a stochastic dynamics $X_t$ in $\R^d$. In the following, we will focus on the overdamped Langevin dynamics:
\begin{equation}\label{eq:eds}
dX_t = -\nabla V (X_t) \, dt + \sqrt{2 \beta^{-1}} dW_t,
\end{equation}
 driven by a potential energy $V: \R^d \to
\R$.  
 However, the algorithm we study equally applies to a Langevin equation or a kinetic 
Monte-Carlo (KMC) dynamics. 

Intuitively, the
dynamics~\eqref{eq:eds} may be seen as the motion of $X_t$ in the energy landscape
defined by the potential~$V$.  Such a landscape typically exhibits many
wells. The process $X_t$ progressively  discovers and successively explores
these wells. The time spent by the process in a well before it hops to another one
might be computationally prohibitively long, thus the need for an alternative numerical approach
 to the direct simulation of the process. The parallel
replica dynamics is such an approach (others are discussed in~\cite{voter-montalenti-germann-02}).

We henceforth assume that $V:\R^d \to \R$ is a smooth potential, such
that $\int \exp(-\beta V) < \infty$. We consider  an application 
 $${\mathcal S}: \R^d \to \N$$
 that associates to a given position $x$  a state ${\mathcal S}(x)$ in a
 discrete space, say $\N$. In practice and for instance, ${\mathcal S}$  maps a position $x$ to the local minimum
reached by the gradient dynamics ($\dot{y} = - \nabla V(y)$) starting
from $x$, and these local minima are numbered (as they are discovered as
the algorithm proceeds). For simplicity, we may think of~${\mathcal S}$ as any discrete valued map. The state-to-state dynamics we will henceforth consider as reference dynamics
 is  $({\mathcal
  S}(X_t))_{t \ge 0}$.

\medskip


 The aim of the parallel replica dynamics is to generate a trajectory
$(S_t)_{t \ge 0}$ which is more efficiently computed than, but shares
the same law as, the original
trajectory $({\mathcal
  S}(X_t))_{t \ge 0}$ obtained from $X_t$, the solution to~\eqref{eq:eds}. Two adjustable
times, $\tau_{dephase}$ and $\tau_{corr}$, will enter the definition of
the dynamics $(S_t)_{t \ge 0}$. In practice, these times may be
state-dependent, but in many practical situations, they are fixed {\em a
  priori} and taken to be equal. One purpose of the analysis below is  to provide a theoretical guideline for the choice of these two times, and in particular $\tau_{corr}$.

\medskip

The parallel replica dynamics as implemented in~\cite{voter-98} consists of the following flow
chart. Consider the initial condition $X^{ref}_0=X_0$ for a reference
walker $(X^{ref}_{t})_{t \ge 0}$. Consider the associated initial condition for the state
dynamics $S_0={\mathcal S}(X_0)$, and set the simulation clock
$T_{simu}$ to zero. Then iterate on the following steps:
\begin{enumerate}
\item {\em Decorrelation step}: Let the reference walker
  $(X^{ref}_{T_{simu}+t})_{t \ge 0}$ evolve according to~\eqref{eq:eds} over a time interval $t \in [0, \tau_{corr}]$. Then
\begin{itemize}
\item If the process leaves the well during this time interval, namely
  if there exists a time $t \le \tau_{corr}$ such that ${\mathcal
    S} \left(X^{ref}_{T_{simu}+t}\right) \neq {\mathcal
    S} \left(X^{ref}_{T_{simu}}\right)$, then
  \begin{itemize}
  \item (i) advance the simulation clock by
  $\tau_{corr}$: $T_{simu}=T_{simu}+\tau_{corr}$, and 
\item (ii) return to 1.
  \end{itemize}
\item If not, then
\begin{itemize}
  \item (i) advance the simulation clock by $\tau_{corr}$:
  $T_{simu}=T_{simu}+\tau_{corr}$, and 
\item (ii) proceed to 2.
\end{itemize}
\end{itemize}
During this step, the state dynamics $S_t$ is defined as:
$$S_t={\mathcal S}(X^{ref}_t)$$
and is thus exact.
\item {\em Dephasing step} : Replicate the walker $X^{ref}_{T_{simu}}$
  into $N$ replicas (in practice $N$ typically ranges between $10^2$ and $10^4$), that is, for $k
  \in \{1, \ldots, N\}$, set:
$$X^k_{T_{simu}}=X^{ref}_{T_{simu}}.$$
Let these replicas evolve independently (according to~\eqref{eq:eds}
using independent Brownian motions) over a time interval $t \in [0,
\tau_{dephase}]$ and using the following rule: if one of the replica
(say $K$) leaves the well during the time interval
$[0,\tau_{dephase}]$, then this
particular replica  is eliminated and the dephasing step for this
particular replica is restarted from the initial position
($X^{K}_{T_{simu}}=X^{ref}_{T_{simu}}$). Throughout this step, the
simulation clock is not advanced, nor are the state dynamics $S_t$ updated. With a
slight abuse of notation, we therefore still denote by $X^k_{T_{simu}}$ the positions
of the replicas at the end of this dephasing step.
\item {\em Parallel step} : Let all the replicas evolve independently
  and denote by 
$$T^k_{W}= \inf\{ t \ge 0, \, {\mathcal S}(X^k_{T_{simu}+t}) \neq {\mathcal S}(X^k_{T_{simu}}) \}$$
the first time the $k$-th replica leaves the current well (denoted by $W=\{x,\, {\mathcal S}(x)={\mathcal S}(X^1_{T_{simu}} \}$). Introduce the first observed escape time over all the replicas
$$T=\inf_{k} T^k_{W}$$
 and the index 
$$K_0=\arg\inf_{k} T^k_{W}$$
of the replica that first leaves the well. Note that
$T=T^{K_0}_{W}$. Note also that the probability that many replica may leave simultaneously is zero, so that the index $K_0$ is well-defined. The parallel step is terminated at the first observed escape event,
in which case the simulation clock is advanced by $NT$ (at least for synchronized CPUs, see the discussion at the end of Section~\ref{sec:justif_par}) and the position
of the reference walker is set to the position of the particular replica that just
underwent the transition:
$$T_{simu}=T_{simu} + N T \text{  and   } X^{ref}_{T_{simu}+NT}=X^{K_0}_{T_{simu}+T}.$$
Over the whole time interval $[T_{simu},T_{simu} + NT]$ of length $NT$, the state dynamics $S_t$ is constant and defined as:
$$S_t={\mathcal S}(X^{1}_{T_{simu}}).$$
{\em Return to 1.}
\end{enumerate}


\medskip
  
The question we are interested in is to understand the error intrinsically
present in the numerical approach, namely the
difference between the law of $(S_t)_{t \ge 0}$ as defined by the above
parallel replica dynamics and the law of
$({\mathcal S}(X_t))_{t \ge 0}$, $X_t$ being the solution
to~\eqref{eq:eds}.

\medskip

\medskip

Besides a formalized answer to the above question (see
Proposition~\ref{prop:error} below), the main outcomes of our work are the
following:
\begin{itemize}
\item the aim of the dephasing step of the parallel replica dynamics is to sample  the  so-called \emph{quasi-stationary distribution} of the well currently
  visited by the dynamics; efficient approaches for completing this goal
  include the Fleming-Viot algorithm, which is a small
  variation of the approach originally implemented  in the  dephasing step of the algorithm;
\item the parallel replica dynamics can be proven to be, in a
  mathematical sense, an
  approximation of the original dynamics and this holds
  irrespective whether this original  dynamics is metastable or
  not; of course the
    efficiency of the approach is improved, and the calibration of the adjustable
    parameters is easier, when metastability occurs in a sense made
    precise below in terms of the spectral properties of a certain operator.
\item for the parallel replica dynamics to be maximally efficient, the adjustable parameter  $\tau_{corr}$ can be calibrated in terms
  of characteristic quantities of the original dynamics
  (see Equation~\eqref{eq:adjust} below); the relevant quantities are 
  expressed in terms of eigenvalues of a given operator associated to
  the dynamics, but can be difficult to practically determine for an arbitrary definition of states.
\end{itemize}

\subsection{A first discussion and an outline of the article}

We first notice that if $\tau_{corr}$ is chosen infinite, the parallel step is never activated. The parallel replica dynamics then amounts
to a simple, classical simulation of the dynamics~\eqref{eq:eds}. We are thus interested in situations where $\tau_{corr}$ is chosen finite.

\medskip

For the parallel step not to introduce any error, two essential assumptions are required on the replicas obtained after the dephasing step:
\begin{itemize}
\item [{[H1]}] the initial positions $X^k_{T_{simu}}$ for the parallel step
  are {\em i.i.d.} and
\item  [{[H2]}] conditionally on the past ${\mathcal F}_{T_{simu}}$ (${\mathcal F}_t$ being the filtration generated by the Brownian motions used in the simulation up to time $t$), the
  stopping times $T^k_{W}$ are {\em exponentially distributed} and are {\em independent of the next visited state} (for all $k$, and thus for $k=1$ since we suppose all the $X^k$ initially i.i.d.).
\end{itemize}
 
 \noindent Under these two assumptions (see Section~\ref{sec:justif_par} below), 
\begin{itemize}
\item [{(i)}] $NT$ has the same law as $T^1_{W}$
(where we recall $T= \inf_{k} T^k_{W}$) and
\item [{(ii)}] the next 
state visited by the replica that is the first to undergo a transition has the
same law as the next visited state for one single arbitrary replica.
\end{itemize}
In other words, under assumptions  [H1] and [H2], the parallel step is "exact"
in the sense that it updates the current state into a new state exactly
equal (in law) to the
state reached when one considers only one replica
distributed according to the distribution obtained after the dephasing
step, and waits for the time for this replica to undergo a transition to
a new well. In terms of wall-clock time, the speed-up is of order $N$. This is
the evident practical interest of the algorithm.

Note that a motivation for considering [H2] is that a 
state\--to\--state dynamics $U_t$ is a continuous-time Markov process if and only if it
satisfies the following two conditions:
\begin{itemize}
\item the list of visited states denoted by
$$(\overline{U}_1,\overline{U}_2, \ldots, \overline{U}_n, \ldots)$$
 is a Markov chain (a discrete-time and discrete-space Markov process) and
\item the times successively spent in each state, denoted by $$(H_1,H_2, \ldots, H_n, \ldots),$$
(namely $U_t=\overline{U}_1$ for $t \in [0, H_1)$, $U_t=\overline{U}_i$ for $t \in [H_1+ \ldots + H_{i-1}, H_1 + \ldots + H_i)$),
are such that (i) the law of $H_i$ given $\overline{U}_i$ is exponential
and (ii) conditionally on  $\overline{U}_i$, the time $H_i$ spent in the
well  and the next well visited $\overline{U}_{i+1}$ are {\em independent} random variables.
\end{itemize}
Thus, if ${\mathcal S}(X_t)$ was a Markov process, the algorithm
would be exact, and the decorrelation step would not be needed. Each cycle of the decorrelation step can thus be seen as a test of the Markov character of ${\mathcal S}(X_t)$, in that, upon successful decorrelation, 
the system is deemed to be a proper starting point for a subsequent parallel stage. One may see our analysis as a way to quantify 
the error introduced by this assumption.

\medskip

After a few remarks on the underlying dynamics in the next section, our
work is organized as follows. In Section~\ref{sec:dephasing}, we analyze
the dephasing step. Then, Section~\ref{sec:justif_par} is devoted to the
parallel step. Finally, our main result is presented in
Section~\ref{sec:decorrelation}, where we analyze the error introduced by the decorrelation step.

\subsection{Remarks on the reference dynamics}
\label{ssec:about}

\subsubsection{Overdamped Langevin and recrossing events}

The algorithm as described above may actually look weird for the
continuous-in-time overdamped dynamics~\eqref{eq:eds}. Indeed, the first
time the process leaves a given state is immediately posterior to the
time it entered that state (this phenomenon is called
\emph{recrossing}). Thus, after a parallel step,
the
process \emph{cannot} remain in the new visited state during the first correlation time interval: the first decorrelation step is always unsuccessful for such a dynamics. One simple way to overcome
this difficulty is to let the reference walker evolve for a fixed small amount of
time after the parallel step, before proceeding  to the next
decorrelation step. This allows the process to leave the vicinity of the
boundary of the new visited state. Another way to deal with this difficulty would be to change the decorrelation step as: Let the reference walker evolve according to~\eqref{eq:eds} over a time $t$ that is the minimum time such that there is no transition over $[t-\tau_{corr},t]$.

\subsubsection{Generalization to other dynamics}

We would like to mention that our analysis carries over to kinetic Monte Carlo models, namely for a pure jump Markov process valued in a finite state space. In this case, the map ${\mathcal S}$ reduces the original discrete state space to a coarser one.

On the other hand, it is unclear how to generalize our study to a Langevin dynamics:
\begin{equation}\label{eq:lang}
\left\{
\begin{aligned}
dq_t &= M^{-1} p_t \, dt, \\
dp_t &= -\nabla V (q_t) \, dt - \gamma M^{-1} p_t \, dt + \sqrt{2 \gamma \beta^{-1}} dW_t,
\end{aligned}
\right.
\end{equation}
since the underlying elliptic degenerate infinitesimal generator causes
additional difficulties for the spectral analysis. Notice that the algorithm itself however readily applies to such a dynamics.


\section{The quasi-stationary distribution and the formalization of the dephasing step}\label{sec:dephasing}

To start with, we discuss here the dephasing step. As mentioned above (see [H1] and
[H2]), the purpose of this step  is to generate {\em independently distributed}
initial conditions for the parallel step, and to complete this according
to a {\em distribution such that the escape time is exponentially
  distributed and independent of the next visited state}. We now explain
here how to create, to some extent, an ideal dephasing step that
satisfies both conditions [H1] and [H2]. The main ingredient of our
formalization is the notion of
\emph{quasi-stationary distribution}, henceforth abbreviated as QSD.

\subsection{The quasi-stationary distribution}\label{sec:QSD}

Consider a state $W \subset \R^d$, and let
$$T^x_W = \inf \{ t > 0, X^x_t \not \in W \}$$
be the first escape time of $W$
for the stochastic process $X^x_t$ satisfying~\eqref{eq:eds} and
starting at $x\in W$ at time $0$. 
The state $W$ is in practice a level set of the map ${\mathcal S}$, and
we suppose in the following that $W$ is fixed, and is a bounded
Lipschitz domain of $\R^d$.  A quasi-stationary
distribution  $\nu$, for the stochastic process $X_t$ and associated to
$W$, is a distribution with support in $W$ and such that, for any
positive time $t$ and for any measurable set $A \subset W$,
\begin{equation}\label{eq:QSD}
\nu (A) = \frac{\displaystyle \int_W \P(X_t^x \in A , \, t < T_W^x) \, d\nu }{\displaystyle \int_W \P( t < T_W^x) \, d\nu}.
\end{equation}
In words, if $X_0$ is distributed according to $\nu$, then, conditionally on not having left the well $W$ up to time $t$, $X_t$ is still distributed according to $\nu$. 

For the convenience of the reader, we collect in this section a few
elementary properties of the QSD. For more details on the theory, we
refer, for example, to~\cite{cattiaux-collet-lambert-martinez-meleard-san-martin-09,martinez-san-martin-04,mandl-61,collet-martinez-san-martin-95,steinsaltz-evans-07,pinsky-85,ferrari-kesten-martinez-picco-95,ferrari-maric-07,ferrari-martinez-san-martin-96}.

Let $X_t$ be the stochastic process satisfying~\eqref{eq:eds}. We introduce its infinitesimal generator:
$$L= - \nabla V \cdot \nabla + \beta^{-1} \Delta,$$ 
and we denote by $L^*=\div(\nabla V \cdot) + \beta^{-1} \Delta$ its adjoint.

We start by stating a Feynman-Kac formula that will be useful below.
\begin{proposition}\label{prop:FK}
Consider a smooth solution $v(t,x)$ to the problem:
$$
\left\{
\begin{aligned}
\partial_t v &= L v \text{ for $t \ge 0$, $x \in W$, }\\
v&= \varphi \text{ on $\partial  W$,} \\
v(0,x)&=v_0(x),
\end{aligned}
\right.
$$
where $\varphi$ is a smooth function.
Then, $$v(t,x) = \E \left(1_{T_W^x < t}\, \varphi(X_{T_W^x}^x) \right) + \E \left(1_{T_W^x \ge t}\, v_0(X_{t}^x) \right),$$
where $X^x_t$ is the process starting at $x$ at time $0$ and $T_W^x$ the first exit time from~$W$.
\end{proposition}
\begin{proof}
Fix a time $t$ and consider $u(s,x)=v(t-s,x)$, which satisfies
$$
\left\{
\begin{aligned}
\partial_s u +L u&= 0 \text{ for $s \in [0,t]$, $x \in W$, }\\
u&= \varphi \text{ on $\partial W$,} \\
u(t,x)&=v_0(x).
\end{aligned}
\right.
$$
Using It\^o calculus, we see that: $\forall s \in [0, T_W^x \wedge t]$,
\begin{align*}
u(s,X_s^x) &= u(0,x) + \int_0^s (\partial_s u + L u) (r,X_r^x) \, dr + \sqrt{2 \beta^{-1}} \int_0^s \nabla u(r,X_r^x) \, dW_r \\
&= u(0,x) + M_s,
\end{align*}
where $M_s=\sqrt{2 \beta^{-1}} \int_0^s \nabla u(r,X_r^x) \, dW_r$ is a
local martingale. Since $u$ is assumed to be smooth, and $X_r^x$ lives
in the bounded domain $W$ up to time $T_W^x \wedge t$, we conclude:
\begin{align*}
v(t,x)=u(0,x) &= \E \left( u(t \wedge T_W^x ,X_{t \wedge T_W^x}^x) \right) \\
 &= \E \left( 1_{T_W^x < t } \, u(T_W^x ,X_{T_W^x}^x) \right) + \E \left( 1_{T_W^x \ge t } \, u(t  ,X_{t}^x) \right) \\
 &= \E \left( 1_{T_W^x < t } \, \varphi(X_{T_W^x}^x) \right) + \E \left( 1_{T_W^x \ge t } \, v_0(X_{t}^x) \right).
\end{align*}
\end{proof}

The quasi-stationary distribution is related to spectral properties of the generator $L$ supplemented with zero Dirichlet boundary conditions on $\partial W$. Let us make this precise. We introduce the invariant measure for the dynamics $X_t$:
$$\,d\mu=Z^{-1} \exp(-\beta V(x)) \, dx$$
where $Z=\int_{\R^d} \exp(-\beta V)$. It is well known that the dynamics~\eqref{eq:eds} is reversible with respect to~$\mu$: for all smooth test functions $f$ and $g$,
$$\int_{\R^d} f \,L g \, d \mu = \int_{\R^d} g \,L f \, d \mu = - \beta^{-1} \int_{\R^d} \nabla f \cdot \nabla g \, \,d\mu.$$
This in turn implies that the dynamics restricted to $W$ is
reversible with respect to $\mu$ restricted to $W$, that is: for all
smooth test functions $f$ and $g$ vanishing on~$\partial W$,
$$\int_W f\, L g \, d \mu = \int_W g\, L f \, d \mu= - \beta^{-1} \int_W \nabla f \cdot \nabla g \, \,d\mu.$$
Thus, the operator $L$ with Dirichlet boundary conditions on $\partial
W$ is negative-definite and symmetric with respect to the scalar product 
\begin{equation}\label{eq:scal_prod}
\langle f , g \rangle_\mu = \int_W f g \, d \mu.
\end{equation}
We denote by $L^2_\mu$ the Hilbert space of functions from $W$ to $\R$ which are square integrable with respect to $\mu$, equipped with the scalar product~\eqref{eq:scal_prod}.
 Since $V$ is assumed to be smooth, the inverse of the operator $L$ from
 $L^2_\mu$ to $L^2_\mu$ is compact, and we thus introduce its
 eigenvalues $(-\lambda_1,-\lambda_2,\ldots,-\lambda_n,\ldots)$ counted
 with multiplicity:
\begin{equation}\label{eq:vp}
0 > - \lambda_1 > - \lambda_2 \ge \ldots \ge - \lambda_n \ge \ldots
\end{equation}
and the associated eigenfunctions $$(u_1,u_2,\ldots,u_n,\ldots)$$ which we
assume  normalized: $\int_W |u_n|^2 \,d\mu = 1$. 
Note that the kernel of $L$ is reduced to $0$ (so that $\lambda_1>0$ in~\eqref{eq:vp}). Using the fact that
\begin{equation}\label{eq:vecp}
\lambda_1 = \inf_{f \in H^1_{\mu,0}} \frac{ \beta^{-1} \int_W |\nabla f|^2 \, d \mu}{\int_W f^2 \, \,d\mu}
\end{equation}
(where $H^1_{\mu,0}$ denotes the space of functions such that $\int |\nabla f|^2 + f^2 \, \,d\mu < \infty$ which vanish on~$\partial W$)
it follows by a standard argument  (if $u_1$ is a minimizer, then
$|u_1|$ is also a minimizer) that we may always assume that $u_1$ is a signed, say nonnegative function. Using the
Harnack inequality, it is again standard to show $u_1$ does not
cancel on $W$. We therefore have  
$$u_1>0 \text{ on W},$$
while $u_1$ vanishes on~$\partial W$. This in turn shows  that
$\lambda_1$ is non-degenerate (thus $\lambda_2>\lambda_1$
in~\eqref{eq:vp}) and that the function $u_1$ is the only signed
eigenfunction. For these standard arguments, we {\it e.g.}  refer to~\cite[Section 8.12]{gilbarg-trudinger-77}).

We now introduce the probability measure
\begin{equation}\label{eq:nu}
d\nu=\frac{u_1 \,d\mu}{\int_W u_1 \,d\mu }
\end{equation}
 on $W$. It is standard that $\nu$ is indeed a QSD:
\begin{proposition}
\label{prop:eigenvector}
The measure $\nu$ defined by~\eqref{eq:nu} is a QSD, that is,
satisfies~\eqref{eq:QSD}. In addition, $\nu$~is the eigenfunction associated with the eigenvalue $-\lambda_1$ for the
Fokker-Planck operator $L^*$ with homogeneous Dirichlet (also known as absorbing) boundary conditions. More precisely, if we denote by $w=\frac{d\nu}{dx}=u_1 \exp(-\beta V) / ( Z \int_W u_1 d \mu)$ the density of $\nu$ with respect to the Lebesgue measure, we have
\begin{equation}\label{eq:vpL*}
\left\{
\begin{aligned}
L^* w &= - \lambda_1 w \text{ on $W$,}\\
w & = 0 \text{ on $\partial W$.}
\end{aligned}
\right.
\end{equation}
The eigenvalue $-\lambda_1$ is the first eigenvalue of $L^*$, and is non-degenerate. 
\end{proposition}
\begin{proof}
To get~\eqref{eq:QSD}, it is sufficient to prove that for any smooth function $f$ vanishing on~$\partial W$:
\begin{equation}\label{eq:1}
\int_W \E(f(X_t^x) \, 1_{t \le T^x_W}) \, d\nu = \int_W f d \nu \,  \int_W \P(t \le T^x_W) \, d\nu.
\end{equation}
Denote by $v(t,x)=\E(f(X_t^x) \, 1_{t \le T^x_W})$ and $\overline{v}(t,x)=\P(t \le T^x_W)$. It follows from Proposition~\ref{prop:FK} that
$$
\left\{
\begin{aligned}
\partial_t v &= L v \text{ for $t \ge 0$, $x \in W$, }\\
v&= 0 \text{ on $\partial  W$,} \\
v(0,x)&=f(x),
\end{aligned}
\right.
$$
and $\bar v$ satisfies the same equation with initial condition $\bar
v(0,x)=1$. Thus, we get: 
\begin{align*}
\frac{d}{dt}\int \E(f(X_t^x) \, 1_{t \le T^x_W}) \, d\nu 
&=\frac{d}{dt}\int v(t,x) u_1(x) \, \,d\mu \left( \int_W  u_1\, d\mu \right)^{-1} \\
&=\int L v(t,x) u_1(x) \, \,d\mu  \left( \int_W u_1 \, d\mu \right)^{-1}\\
&=\int  v(t,x) L u_1(x) \, \,d\mu \left( \int_W u_1 \, d\mu \right)^{-1}\\
&=- \lambda_1 \int  v(t,x)  u_1(x) \, \,d\mu \left( \int_W u_1 \, d\mu \right)^{-1}\\
&=- \lambda_1 \int \E(f(X_t^x) \, 1_{t \le T^x_W}) \, d\nu.
\end{align*}
This implies,
$$\int \E(f(X_t^x) \, 1_{t \le T^x_W}) \, d\nu=\int f d \nu \exp(-\lambda_1 t)$$
which in turn yields~\eqref{eq:1} similarly arguing on~$\overline{v}$.

The relation between the spectrum of the operator $L$ with Dirichlet boundary conditions on $\partial W$ seen as on operator on $L^2_\mu$, and the operator $L^*$ with absorbing boundary conditions, follows the variational equality satisfied by the functions $u_k$: for all test function $f \in H^1_{\mu,0}$,
$$-\int_W u_k L f  \, d \mu =\beta^{-1} \int_W \nabla u_k \cdot \nabla f \, d \mu = \lambda_k \int_W u_k f \, d \mu,$$
which shows that $(-\lambda_k,u_k \exp(-\beta V))$ is an eigenvalue / eigenfunction couple for $L^*$. The converse is obtained similarly.
\end{proof}

%

\begin{remark}
  It will follow from Proposition~\ref{prop:error} below that there is
actually a \emph{unique} QSD on~$W$. We will indeed prove there the
(actually exponentially fast) long-time convergence to the QSD $\nu$ for the
process $X_t$ conditioned to stay in $W$, irrespective of the initial distribution.
\end{remark}

The main proposition of this section is the following:
\begin{proposition}\label{prop:QSD}
Consider the quasi-stationary distribution $\nu$ associated to the dynamics~\eqref{eq:eds} on $X_t$, and defined by~\eqref{eq:nu}.
Then, if $X_0$ is distributed following $\nu$, the first exit time $T_W$
from $W$ is exponentially distributed and is a random variable independent of the first hitting point on $\partial W$.
\end{proposition}
\begin{proof}
Consider, for a smooth test function $\varphi: \partial W \to \R$, $v$ solution to:
$$
\left\{
\begin{aligned}
\partial_t v &= L v \text{ for $t \ge 0$, $x \in W$, }\\
v&= \varphi \text{ on $\partial  W$,} \\
v(0,x)&=0.
\end{aligned}
\right.
$$
We know from Proposition~\ref{prop:FK} that, for all $t\ge 0$ and $x \in W$,
$$v(t,x)= \E \left(1_{T_W^x < t}  \, \varphi\left(X_{T_W^x}^x\right) \right).$$
Consider now
$$f(t) = \int_W v(t,x) d\nu = \E^\nu \left(1_{T_W < t}\, \varphi(X_{T_W}) \right),$$
where the superscript $\nu$ indicates that the process $X_t$ we consider
is assumed to start at $t=0$ under
the quasi-stationary distribution~$\nu$. Then, we have :
\begin{align*}
f'(t)
&=\int_W \partial_t v(t,x) d\nu \\
&=\int_W L v(t,x) d\nu \\
&=\int_W \left( -\nabla V \cdot \nabla v + \beta^{-1} \Delta v \right) d\nu\\
&=\int_W \left( v \, \div (\nabla V \nu) - \beta^{-1}  \nabla v \cdot  \nabla \nu \right) - \int_{\partial W} \nabla V  \cdot n \, \nu + \beta^{-1} \nabla v \cdot n \, \nu\\
&= \int_W v \, \left( \div (\nabla V \nu) +  \beta^{-1} \Delta \nu  \right) - \beta^{-1} \int_{\partial W}  v  \nabla \nu \cdot n \\
&= \int_W v L^* \nu - \beta^{-1} \int_{\partial W}  \varphi \,  \nabla \nu \cdot n \\
&= - \lambda_1 f  + \lambda_1 \int_{\partial W}  \varphi \, d \rho.
\end{align*}
where $n$ denotes the outward normal to $W$,
where  $$\rho(dx)=-\frac{\beta^{-1}}{\lambda_1} \nabla \nu \cdot   n \,
d\sigma_{\partial W}$$ $\sigma_{\partial W}$ denotes the Lebesgue
measure on $\partial W$, and where, with a slight abuse of notation, we denote by $\nu$ both the probability measure and its density  with respect to the Lebesgue measure on $W$, which is proportional to $u_1(x) \exp(-\beta V(x))$. This yields
$$ \E^\nu \left(1_{T_W < t} \, \varphi(X_{T_W}) \right) = f(t)=\left( 1 - \exp(-\lambda_1 t) \right) \int \varphi \, d \rho,$$
which concludes the proof, $\rho$ being then the first hitting point distribution on $\partial W$.
\end{proof}

A natural question for practical purposes is whether the fact that the exit time has an exponential law implies that the initial condition is distributed according to the QSD. Here is a necessary and sufficient condition.
\begin{proposition}
Let us assume that $X_t$ is solution to~\eqref{eq:eds}, with an initial condition $X_0$ distributed according to a probability measure $\mu_0$ with support in the well $W$ and such that:
\begin{equation*}
\int_W \left(\frac{d \mu_0}{d \mu} \right)^2 d \mu < \infty.
\end{equation*}
Let us assume that the first exit time from $W$ (denoted by $T_W$) is exponentially distributed.

The initial distribution is necessarily the QSD ($\mu_0=\nu$) if and only if the eigenvalues  of the operator $L$ on $L^2_\mu$ with zero Dirichlet boundary conditions are non-degenerate (see~\eqref{eq:vp}): $\forall i \ne j, \, \lambda_i \ne \lambda_j$ and  $\forall k \ge 2$,
$$\int_W u_k d \mu \neq 0,$$
where $u_k$ denotes the $k$-th eigenfunction (see~\eqref{eq:vecp}). 
\end{proposition}
\begin{proof}
The proof is divided into three steps.

\medskip
\noindent{\em Step 1: A rewriting of $\P(T_W \ge t)$ in terms of the eigenvalues and eigenfunctions of $L$.}

The assumption on $T_W$ is equivalent to the fact that there exists a positive $\lambda$ such that, for all time $ t\ge 0$,
$$\P(T_W \ge t) =\exp(- \lambda t).$$
Using the same reasoning as in the proof of Proposition~\ref{prop:QSD}, the left hand-side can be rewritten as: for all time $t \ge 0$
\begin{equation}\label{eq:2}
\P(T_W \ge t) = \int v(t,x) \mu_0(dx),
\end{equation}
where $v$ satisfies the partial differential equation:
$$
\left\{
\begin{aligned}
\partial_t v &= L v \text{ for $t \ge 0$, $x \in W$, }\\
v&= 0 \text{ on $\partial  W$,} \\
v(0,x)&=1.
\end{aligned}
\right.
$$
Using the spectral decomposition of the operator $L$ with homogeneous Dirichlet boundary conditions on $\partial W$, we have:
$$v(t,x)=\sum_{k \ge 1} \exp(-\lambda_k t) \left( \int_W u_k d \mu \right) u_k(x).$$
This equality holds for example in the functional space ${\mathcal C}(\R_+,L^2_\mu)$. Using this decomposition in~\eqref{eq:2}, one gets
\begin{equation}\label{eq:3}
\exp(- \lambda t)=\P(T_W \ge t) = \sum_{k \ge 1} \exp(-\lambda_k t) \left( \int_W u_k d \mu \right) \left( \int_W u_k d\mu_0 \right),
\end{equation}
which holds for all time $t \ge 0$. Notice that the convergence of the series is normal, for the $L^\infty$-norm on $t$, since by Cauchy Schwarz,
$$\sum_{k \ge 1} \left| \int_W u_k d \mu \right| \left| \int_W u_k d\mu_0 \right| < \infty.$$

\medskip
\noindent{\em Step 2: One implication.}

Let us assume that the eigenvalues of the operator $L$ are non-degenerate and that $\forall k \ge 2, \int_W u_k d \mu \neq 0$.
Thus, in the limit $t \to \infty$, the right hand-side of~\eqref{eq:3} is equivalent to the first term of the series (since $\lambda_1$ is non-degenerate):
$$\sum_{k \ge 1} \exp(-\lambda_k t) \left( \int_W u_k d \mu \right) \left( \int_W u_k d\mu_0 \right) \sim \exp(-\lambda_1 t) \left( \int_W u_1 d \mu \right) \left( \int_W u_1 d\mu_0 \right).$$
Using now~\eqref{eq:3}, this implies that
\begin{equation}\label{eq:k=1}
\lambda_1=\lambda \text{ and }  \left( \int_W u_1 d \mu \right) \left( \int_W u_1 d\mu_0 \right)=1.
\end{equation}
Subtracting $\exp(-\lambda t)$ from both sides of~\eqref{eq:3}, and repeating the argument, one gets that for all $k \ge 2$,
\begin{equation}\label{eq:kge2}
\left( \int_W u_k d \mu \right) \left( \int_W u_k d\mu_0 \right)=0.
\end{equation}
which implies that: $\forall k \ge 2$, 
$$\int_W u_k d \mu_0 = 0.$$
Thus, $\frac{d\mu_0}{d\mu}$ only has a component along the first eigenfunction $u_1$, which implies (using~\eqref{eq:k=1}):
$$d\mu_0=\frac{d\mu_0}{d\mu} d\mu = \frac{u_1 d\mu}{\int_W u_1 d\mu} = d\nu.$$
The initial condition is necessarily the QSD.

\medskip
\noindent{\em Step 3: The other implication.}

Conversely, let us assume that: $\exists k_0 \ge 2$,
$$\int_W u_{k_0} d \mu = 0.$$
Let us then consider the measure $\mu_0$ defined by
$$d\mu_0 = \left( \frac{u_1}{\int_W u_1 d\mu} + \varepsilon u_{k_0}  \right) d \mu.$$
Clearly, $\int_W d\mu_0=1$ and $\mu_0$ is a non-negative measure for $\varepsilon>0$ sufficiently small (using the regularity of the eigenfunctions on $\overline{W}$). Thus, $\mu_0$ is a probability measure which is such that~\eqref{eq:3} is satisfied,
and thus an initial condition for~\eqref{eq:eds} which is different from the QSD, but such that the exit time is exponentially distributed.

Likewise, let us assume that one eigenvalue is degenerate: $\exists k_0 \ge 2$, $$\lambda_{k_0}=\lambda_{k_0+1}.$$
Let us then consider the measure $\mu_0$ defined by
$$d\mu_0 = \left( \frac{u_1}{\int_W u_1 d\mu} + \varepsilon \left( \left( \int_W u_{k_0+1} \, d\mu \right) u_{k_0} - \left( \int_W u_{k_0} \, d\mu \right) u_{k_0+1} \right) \right) d \mu.$$
Again, $\int_W d\mu_0=1$ and $\mu_0$ is a non-negative measure for $\varepsilon>0$ sufficiently small. Thus, $\mu_0$ is a probability measure which is such that~\eqref{eq:3} is satisfied, and thus an initial condition for~\eqref{eq:eds} which is different from the QSD, but such that the exit time is exponentially distributed.
\end{proof}

A simple example of a situation where there exists a $k_0 \ge 2$ such that $\int_W u_{k_0} d \mu = 0$ is the following: $W=(0,1)^d$ and $V=0$ on $W$ (so that $\mu$ is simply the Lebesgue measure on $W$). Some eigenfunctions of the Dirichlet laplacian operator on $(0,1)^d$ indeed have zero mean. In dimension $d=1$, one can for example consider the probability measure $\mu_0$ with density proportional to $\frac{\sin(\pi x)}{\int_0^1 \sin(\pi x) \, dx} + \varepsilon \sin(2 \pi x)$ (which is different from the QSD $\frac{\sin(\pi x)}{\int_0^1 \sin(\pi x) \, dx} \, dx$) to obtain exponentially distributed exit times from $(0,1)$.

\subsection{Formalization of the dephasing step}

As stated in Proposition~\ref{prop:QSD}, the crucial property of the QSD
$\nu$ is that if the process starts under~$\nu$, then the exit time from
$W$ is exponentially distributed, and the hitting point on $\partial W$
is independent from the exit time. The ideal dephasing step would
therefore ensure that the replicas are independent and all share the QSD
as  law. Then, [H1] and [H2] would be fulfilled and the parallel step would be {\em exact}, as made precise below in Proposition~\ref{prop:par_step}.

The actual  dephasing step, as implemented in the current version of the
algorithm, can thus be interpreted as an approximation procedure for the
QSD of the well. It is consequently interesting to point out that
basically two techniques are known in the literature to sample the QSD $\nu$.
One method (called the Fleming-Viot method~\cite{burdzy-holyst-march-00,grigorescu-kang-04,ferrari-maric-07,lobus-08}) consists in launching
a set of replicas in $W$, and when one of them leaves the well, to
duplicate one of the other replicas. Then, one lets the number of
replicas and the time go to infinity. In this limit, a finite fixed subset of replicas is i.i.d. with law the QSD. This method is very close to what is
performed during the dephasing step in the original version of the algorithm presented in the introduction. The only slight modification is that the Fleming-Viot method consists in duplicating a replica when one leaves the well, rather than starting again the whole procedure from a fixed initial position.

Another approach consists in considering only one walker in the well,
and each time this walker leaves the well, redistribute it according
to the empirical measure within the well up to the exit time. Again, one
has to consider the distribution in the long-time limit to get the
QSD, see~\cite{aldous-flannery-palacios-88}. This somehow justifies the intuition used in the decorrelation step that, if the process remains for a very long time in a well, it will be distributed according to the QSD, see Section~\ref{sec:decorrelation} below.

\medskip

In summary, it is reasonable, with a view to globally analyze the parallel
replica dynamics, to first replace  the dephasing step by an
\emph{ideal} dephasing step, which consists in instantaneously drawing
$N$ initial positions for the replicas, independently and  according to the QSD.
The issue of generating that particular distribution, either using a
dedicated approach, or precisely using the dephasing step (as currently implemented), is a separate
issue from analyzing the error introduced by the parallel replica dynamics.

\section{Analysis of the parallel step}
\label{sec:justif_par}

Our analysis of the parallel step is formalized in the following proposition, which shows that the parallel step does not introduce any additional error if the assumptions [H1] and [H2] are satisfied.
\begin{proposition}\label{prop:par_step}
Consider $N$ i.i.d. stochastic process  $X_t^k$, their escape times 
$$T^k_W= \inf \{t > 0, X_t^k \not \in W \}$$
from a bounded domain $W$, and the \emph{first} escape time over all processes
$$T=T^{K_0}_W \text{ where } K_0 = \arg \min_{k \in \{1,\ldots,N\}} T^k_W.$$
\begin{itemize}
\item Assume that
$$\text{$T^1_W$ is  exponentially distributed}.$$
Then
$$\text{$NT$ has the same law as $T^1_W$.}$$
\item Assume that
$$\text{$T^1_W$ is independent of the first hitting point on $\partial W$}.$$
Then the first hitting point for $X_t^{K_0}$ on $\partial W$ has the same distribution as  the first hitting point for $X_t^1$ and is independent of $T^{K_0}_W$.
\end{itemize}
\end{proposition}
\begin{proof}
The first statement is standard. If $T^1_W$ is exponentially distributed, then
$$\varphi(t) = \P( T^1_W > t) = \exp(-\lambda t),$$
where $\lambda$ is the parameter of the exponential distribution of
$T^1_W$. Thus,considering $T=\min_{k \in \{1,\ldots,N\}} T^k_W$, we have
\begin{align*}
\P(T > t)
&= \P\left(\min_{k \in \{1,\ldots,N\}} T^k_W > t\right) \\
&= \P\left(\forall k \in \{1,\ldots,N\}, \, T^k_W > t\right) \\
&= \prod_{k=1}^N \varphi(t) \\
&= \exp(- N \lambda t).
\end{align*}
This shows that $T$ is exponentially distributed, with parameter $N
\lambda$. Consequently, $NT$ is exponentially distributed with parameter $\lambda$.

For the second assertion, the assumption can be written as: for all test functions~$f: \partial W \to \R$,
$$\E \left(f \left(X_{T^1_W}^1\right) 1_{T^1_W > t}\right)= \left( \int_{\partial W} f d \rho \right) \varphi(t),$$
where $\varphi(t)=\P(T^1_W > t)$ and  $\rho$ is the first hitting point distribution, with support on $\partial W$. Then, we have
\begin{align*}
\E \left(f \left(X_{T^{K_0}_W}^{K_0}\right) 1_{T^{K_0}_W > t} \right)
&= \sum_{k=1}^N \E \left(f \left(X_{T^{K_0}_W}^{K_0}\right) 1_{T^{K_0}_W > t} 1_{K_0=k}\right) \\
&= \sum_{k=1}^N \E\left(f \left(X_{T^{k}_W}^{k}\right) 1_{T^{k}_W > t} \prod_{l \neq k}1_{T^l_W > T^k_W} \right) \\
&= \sum_{k=1}^N \E\left(f \left(X_{T^{k}_W}^{k}\right) 1_{T^{k}_W > t} \, [\varphi(T^k_W)]^{N-1} \right) \\
&= N \left( \int_{\partial W} f d\rho \right) \int_t^\infty [\varphi(s)]^{N-1} (-\varphi')(s) ds \\
&= \left( \int_{\partial W} f d\rho \right) [\varphi(t)]^{N}.
\end{align*}
This shows that the first hitting point on $\partial W$ for $X^{K_0}_t$ is distributed according to $\rho$, and is independent of~$T^K_{W_0}$.
\end{proof}

Three remarks are in order.

\medskip

First, in the first assertion of Proposition~\ref{prop:par_step}, the fact that $NT$ has the same law as $T^1_W$ is
actually \emph{equivalent} to $T^1_W$ being exponentially distributed. The
former assertion indeed implies the functional equation: $\forall t >0$ and $\forall N \in \N$
$$[\varphi(t/N)]^N = \varphi(t),$$
where $\varphi(t) = \P ( T^1_W > t )$, the only solution to which is
the exponential function.

Second, without the assumption made in the second assertion, the first
  hitting point on $\partial W$ for $X^{K_0}_t$ cannot generically have the same
  distribution as for $X^1_t$. Indeed, if the first hitting point on
  $\partial W$ and the exit time from $W$ are coupled, the distribution
  of $X^{K_0}_{T^{K_0}_W}$ would favor points on the boundary attained
  in shorter times, compared to the distribution of
  $X^{1}_{T^{1}_W}$. This issue is a
  separate issue from that of having or not an exponential distribution for the exit time. The fact that the first hitting point distribution on $\partial W$ for $X_t^{K_0}$ is the same as for $X^1_t$ implies that the next visited state is the same for the two processes.

Finally and as shown by A.F.~Voter in the original article~\cite{voter-98}, the
algorithm does not require synchronized processors to be used in practice, as
would suggest the schematic presentation of the parallel step we give
above. The parallel step above indeed assumes that the processors as
synchronized, since $T$ and $K_0$ are defined in terms of the first
replica which leaves the well, considering the same physical time unit
for all replicas. If the processors are not synchronized, the parallel step is
still exact by considering the first observed replica which leaves the
well, and by advancing the  simulation time by the sum of the physical
times elapsed on each processor, instead of $N T^{K_0}_W$. We now justify this.

Assume that, for $n \in \{2,\ldots,N\}$, the $n$-th processor is
$\rho_n$ times as  fast as  the first one. Then, $\tau^i_W$ (which is
the time needed for the $i$-th replica, run on the $i$-th processor, to leave
the well $W$) is exponentially distributed with parameter $\rho_i
\lambda$. Then,  consider $\tau= \min_{1 \le n \le N} \tau^i_W$
the random time associated to the first detected event. One can check
that $(1+\rho_2+\ldots+\rho_n) \tau$ has the same law as
$\tau^1_W$. This means that advancing the simulation time by the sum of
the (physical times) counted on each processor at the end of the parallel step
is a correct approach.

This reasoning generalizes to non-constant in time processor
speeds. Assume for simplicity that we have $N=2$ processors, and that
the speed of the second processor, compared to the first, is $\rho_2(t)$
(where $\rho_2$ is deterministic and with values in $(0,+\infty)$, and
$t$ is in the time-unit of the first processor). $\tau^1_W$ is
exponentially distributed with parameter $\lambda$: $\P(\tau^1_W \ge
t)=\exp(-\lambda t)$. 
The time $t$ is measured in time-unit of the first processor. So, when
the time is $t$ on the first processor, the time is $R_2(t)=\int_0^t
\rho_2(s)\,ds$ on the second processor. Thus, in the time-unit of the
first processor, $\tau^2_W$ (which is the first time, in the time-unit
of the first processor, an event occur on the second processor)  is the
image by $R_2$ of an exponential law with parameter $\lambda$ :
$\P(\tau^2_W \ge t)=\exp(-\lambda R_2(t))$. Thus $\tau^2_W$ is not
exponentially distributed anymore. 
Consider however $\tau=\min(\tau^1_W,\tau^2_W)$. We have $\P(\tau \le t) = \P(\tau^1_W \le t) \P(\tau^2_W \le t) = \exp(-\lambda(t+R_2(t)))$, so that $\tau$ has density $\lambda(1+\rho_2(t))  \exp(-\lambda(t+R_2(t)))$. When an event occurs, one looks at the sum of the time actually spent on each processor, which is $\tau+R_2(\tau)$. And the law of $\tau+R_2(\tau)$ is indeed exponential with parameter $\lambda$ since $\E(f(\tau+R_2(\tau))) = \int f( u + R_2(u) ) \lambda(1+\rho_2(u)) \exp(-\lambda(u+R_2(u)) \, du = \int f(z) \lambda \exp(-\lambda z) \, dz$.

\section{Analysis of the decorrelation step}\label{sec:decorrelation}

The dephasing step has now been formally replaced by independent draws
according to the QSD, and  we have formalized the parallel step. It now remains to analyze  the
error introduced,  at the end of the decorrelation
step, by the replacement of $X^{ref}_t$ by a random position
distributed according to the QSD. Intuitively, it is expected that this instantaneous draw
could at least be justified if there exists a  timescale separation: when the
process enters a new well, and if this new well is indeed a metastable
region for the dynamics, then the process remains in the well sufficiently long to
reach the quasi-stationary distribution of that well, before hopping to another
well. It is the purpose of the
decorrelation step to check that the process indeed remains in the well
for a sufficiently long time. For the decorrelation step to be
successful, we thus need the actual typical time to reach the QSD to be much smaller than
the typical time to hop to another well. In this picture, $\tau_{corr}$
is seen as an approximation of  the time to reach the QSD. The purpose
of this section is to justify this rigorously.

The decorrelation step is essentially a step that may be seen
as a way to control the error associated to the instantaneous redrawing
according to the QSD in the new state. This redrawing is only considered
legitimate if the decorrelation step has been successful, that is, the
process has spent a sufficiently long time in the current well. Any
method that provides a control of this error would be an equally interesting "decorrelation step."





We again consider~$X_t$ solution to~\eqref{eq:eds} with initial condition $X_0 \in
W$, where $W$ (the well) is a bounded domain, subset of the state
space. We denote by $\mu_0$ the (arbitrary) distribution of $X_0$. 
We consider the process in the current well, and  the
joint distribution of the first hitting point on the boundary of the
well and the first exit time
$$T_W=\inf\{ t \ge 0, X_t \not \in W \},$$
when this point is hit.
We first derive from  the Markov character of~$(X_t)_{t \ge 0}$ a
useful formula:
\begin{lemma}\label{lem:markov}
We have, for all (deterministic) times $t$ and for all test functions
$f: \R_+ \times W \to \R$, 
$$\E(f(T_W-t,X_{T_W}) | {T_W} \ge t) = \int_W \E(f(T_W^x,X^x_{T_W^x})) \, {\mathcal L}(X_{t} |  {T_W}  \ge t) (dx),$$
where $X^x_t$ and $T^x_W$ respectively denote the process solution to~\eqref{eq:eds} with initial condition $x$, and its associated first exit time from $W$. In the right-hand side, ${\mathcal L}(X_{t} |  {T_W}  \ge t)(dx)$ denotes the distribution of $X_t$ conditionally on $T_W \ge t$. Otherwise stated,
$$\E(f({T_W}-t,X_{T_W}) | {T_W} \ge t) = \E (F(X_{t}) |   {T_W} \ge t)$$
where 
\begin{equation}\label{eq:F}
F(x)=\E(f(T_W^x,X^x_{T_W^x})).
\end{equation}
\end{lemma}
\begin{proof}
This is equivalent to prove that
$$\E(f({T_W}-t,X_{T_W}) 1_{{T_W} \ge t}) = \E (F(X_{t}) 1_{{T_W} \ge t} ).$$
The result is then obtained by conditioning by ${\mathcal F}_{t}$ (where ${\mathcal F}_t$ is again the filtration generated by the Brownian motions used in the simulation up to time $t$) and using the Markov property:
\begin{align*}
\E(f({T_W}-t,X_{T_W}) 1_{{T_W} \ge t})
&=\E[\E (f({T_W}-t,X_{T_W}) 1_{{T_W} \ge t} |  {\mathcal F}_{t}) ]\\
&=\E[\E (f({T_W}-t,X_{T_W})  |  {\mathcal F}_{t}) 1_{{T_W} \ge t} ]
\end{align*}
and $\E (f({T_W}-t,X_{T_W})  |  {\mathcal F}_{t}) = F (X_{t})$.
\end{proof}

\medskip

Our purpose is now to estimate is the difference in law between the
following two processes: the original process $X_t$ considered above
(starting from the arbitrary initial condition $X_0$),
given that it has spent a sufficiently long time (say $t$) in the
current well, and a similar process starting from the QSD $\nu$ defined
in~\eqref{eq:nu} as initial
distribution. We wish to estimate this
difference in  the limit $t \to \infty$ (which will then, in practice,  be replaced  by $t > \tau_{corr}$).

\medskip

We introduce the error
\begin{equation}
  \label{eq:error}
  e(t)=\left| \E(f({T_W}-t,X_{T_W}) | {T_W} \ge t) - \E^\nu(f({T_W},X_{T_W})) \right|
\end{equation}
where we recall that the superscript $^\nu$ means that, in the second
term (only!), the process  $X_t$ starts at
time $0$ under the quasi-stationary distribution $\nu$ introduced in~\eqref{eq:nu}. 

\medskip

Before we state our main result on this error, we notice that, with $F$
defined by~\eqref{eq:F}, we have
$$\E(F(X_{t}) | {T_W} \ge t) = \frac{\displaystyle\int_W v(t,x) \,d\mu_0}{\displaystyle\int_W \bar v(t,x) \,d\mu_0}$$
since, we recall, $\mu_0$ denotes the law of $X_0$,
$$v(t,x) = \E \left(1_{T_W^x \ge t} \, F(X^x_t) \right)$$
and
$$\bar v(t,x) = \E(1_{T_W^x \ge t}) = \P(T_W^x \ge t).$$

By denoting again $L$ the infinitesimal generator of $(X_t)_{t \ge 0}$, we know from Proposition~\ref{prop:FK} that
$$
\left\{
\begin{aligned}
\partial_t v &= L v \text{ for $t \ge 0$, $x \in W$, }\\
v&= 0 \text{ on $\partial  W$,} \\
v(0,x)&=F(x),
\end{aligned}
\right.
$$
and $\bar v$ satisfies the same equation with initial condition $\bar v(0,x)=1$.

From the spectral decomposition of the operator $L$ introduced in
Section~\ref{sec:QSD}, we therefore get the following expressions for $v$ and $\bar v$:
$$v(t,x) = \sum_{k \ge 1} \exp(-\lambda_k t) \left( \int_W F u_k \, d\mu \right) \, u_k(x)$$
and
$$\bar v(t,x) = \sum_{k \ge 1} \exp(-\lambda_k t) \left( \int_W u_k \, d\mu \right)  u_k(x).$$
Thus, using the definition~\eqref{eq:nu} of the QSD $\nu$, we have 
\begin{align}
\label{eq:pif}
\E(F(X_{t}) | {T_W} \ge t)
& = \frac{\displaystyle\int_W v(t,x) \,d\mu_0}{\displaystyle\int_W \bar v(t,x) \,d\mu_0} \nonumber\\
& = \frac{\displaystyle  \sum_{k \ge 1} \exp(-\lambda_k t) \int_W F u_k \, d\mu \int_W u_k \,d\mu_0 }{\displaystyle  \sum_{k \ge 1} \exp(-\lambda_k t) \int_W u_k \, d\mu \int_W u_k \,d\mu_0}\nonumber \\
&=  \frac{\displaystyle \int_W u_1\, d\mu  \int_W F d\nu  \int_W u_1 \,d\mu_0  + \sum_{k \ge 2} \exp(-(\lambda_k-\lambda_1) t) \int_W F u_k \, d\mu \int_W u_k \,d\mu_0 }{\displaystyle  \int_W u_1 \, d\mu   \int_W u_1 \,d\mu_0+  \sum_{k \ge 2} \exp(-(\lambda_k-\lambda_1) t) \int_W u_k \, d\mu \int_W u_k \,d\mu_0 }\nonumber\\
&=  \frac{\displaystyle   \int_W F d\nu   + \sum_{k \ge 2} \exp(-(\lambda_k-\lambda_1) t) \frac{\int_W F u_k \,d\mu}{\int_W u_1 \, d\mu} \frac{\int_W u_k \,d\mu_0}{ \int_W u_1 \,d\mu_0} }{\displaystyle  1 +  \sum_{k \ge 2} \exp(-(\lambda_k-\lambda_1) t) \frac{\int_W u_k \, d\mu}{\int_W u_1 \, d\mu} \frac{\int_W u_k \,d\mu_0}{ \int_W u_1 \,d\mu_0} }.
\end{align}
Since $u_1 >0$,  we note $\int_W u_1 d \mu_0 > 0$ and $\int_W u_1\,d\mu> 0$.

\medskip

We are now in position to state the main result of this section:
\begin{proposition}
\label{prop:error}
Assume that the initial arbitrary distribution $\mu_0$ of $X_0$ admits a Radon-Nikodym derivative $\displaystyle
\frac{d \mu_0}{d \mu} $ with respect to the invariant measure $\mu$ of
the dynamics $X_t$, such that
\begin{equation}\label{eq:hyp_mu0}
\int_W \left(\frac{d \mu_0}{d \mu} \right)^2 d \mu < \infty.
\end{equation}
Then, there exists a constant $C$ (which depends on $\mu_0$ but not on
$f$)  such that, for all $t \ge \frac{C}{\lambda_2 - \lambda_1}$, the
error $e(t)$ defined in~\eqref{eq:error} satisfies
$$e(t) \le C \|f\|_{L^\infty} \exp(-(\lambda_2 - \lambda_1)t),$$
where $-\lambda_2 < -\lambda_1 < 0$ are the first two eigenvalues of the
operator $L$ on the weighted space $L^2_\mu$. 
\end{proposition}
\begin{proof}
In order to evaluate \eqref{eq:error}, we first write, using
Lemma~\ref{lem:markov}, 
\begin{align*}
e(t)&=\Big| \E(f({T_W}-t,X_{T_W}) | {T_W} \ge t) - \E^\nu(f({T_W},X_{T_W})) \Big| \\
&= \left|\int_W \E(f(T_W^x,X^x_{T_W^x})) \, {\mathcal L}(X_{t} | {T_W} \ge t) (dx) -  \int_W \E(f(T_W^x,X^x_{T_W^x})) \, d\nu \right| \\
&=  \left| \E(F(X_{t}) | {T_W} \ge t) -   \int_W F d\nu \right|,
\end{align*}
where $F$ is defined by~\eqref{eq:F} and the first term in the
right-hand side has just been expressed in~\eqref{eq:pif}.

We therefore have:
\begin{align*}
e(t)&= \left| \frac{\displaystyle   \int_W F d\nu   + \sum_{k \ge 2} \exp(-(\lambda_k-\lambda_1) t) \frac{\int_W F u_k \, d\mu}{\int_W u_1 \, d\mu} \frac{\int_W u_k \,d\mu_0}{ \int_W u_1 \,d\mu_0} }{\displaystyle  1 +  \sum_{k \ge 2} \exp(-(\lambda_k-\lambda_1) t) \frac{\int_W u_k \, d\mu}{\int_W  u_1 \, d\mu} \frac{\int_W u_k \,d\mu_0}{ \int_W u_1 \,d\mu_0} } -  \int_W F d\nu  \right| \\
&=\left| \frac{\displaystyle     \sum_{k \ge 2} \exp(-(\lambda_k-\lambda_1) t) \frac{ \int_W F u_k \, d\mu - \int_W F d\nu  \int_W u_k \, d\mu}{\int_W u_1 \, d\mu} \frac{\int_W u_k \,d\mu_0}{ \int_W u_1 \,d\mu_0}  }{\displaystyle  1 +  \sum_{k \ge 2} \exp(-(\lambda_k-\lambda_1) t) \frac{\int_W u_k \, d\mu}{\int_W u_1 \, d\mu} \frac{\int_W u_k \,d\mu_0}{ \int_W u_1 \,d\mu_0} } \right|.
\end{align*}
Thus
\begin{align}
\label{eq:alpha}
e(t) \le \exp(-(\lambda_2 - \lambda_1)t) \frac{\displaystyle    \sum_{k \ge 2} \frac{ \left| \int_W F u_k \, d\mu \int_W u_k \,d\mu_0 \right|+ \left| \int_W F d\nu  \int_W u_k \, d\mu\int_W u_k \,d\mu_0 \right|}{\int_W u_1 \, d\mu \int_W u_1 \,d\mu_0} }{\displaystyle  \left| 1 +  \sum_{k \ge 2} \exp(-(\lambda_k-\lambda_1) t) \frac{ \int_W u_k \, d\mu}{ \int_W u_1 \, d\mu} \frac{\int_W u_k \,d\mu_0}{ \int_W u_1 \,d\mu_0} \right|}.
\end{align}
Now, we have by Cauchy-Schwarz and using the fact that $\|F\|_{L^\infty} \le  \|f\|_{L^\infty}$,
\begin{align}
\sum_{k \ge 2}  \left| \int_W F u_k \, d\mu \int_W u_k \,d\mu_0 \right|
&\le \sqrt{\sum_{k \ge 2}  \left| \int_W F u_k \, d\mu\right|^2 } \sqrt{\sum_{k \ge 2}  \left| \int_W u_k \,\frac{d\mu_0}{d\mu} d\mu \right|^2 }   \nonumber\\
&\le \sqrt{\int_W F^2 \, d\mu} \sqrt{\int_W \left(\frac{\,d\mu_0}{\,d\mu}\right)^2 d \mu}\nonumber\\
&\le \sqrt{\mu(W)} \, \|f\|_{L^\infty} \sqrt{\int_W \left(\frac{\,d\mu_0}{\,d\mu}\right)^2 d \mu}
\label{eq:beta1}
\end{align}
and, likewise,
\begin{equation}\label{eq:beta2}
 \sum_{k \ge 2} \left| \int_W F d\nu  \int_W u_k \, d\mu\int_W u_k \,d\mu_0 \right| \le \|f\|_{L^\infty} \sqrt{\mu(W)} \sqrt{\int_W \left(\frac{\,d\mu_0}{\,d\mu}\right)^2 d \mu}.
\end{equation}
Arguing similarly on the denominator of~\eqref{eq:alpha}, we obtain
\begin{align}
\label{eq:gamma}
\left|\sum_{k \ge 2} \exp(-(\lambda_k-\lambda_1) t) \frac{\int_W u_k \, d\mu}{\int_W u_1 \, d\mu} \frac{\int u_k \,d\mu_0}{ \int u_1 \,d\mu_0} \right|
& \le \frac{\exp(-(\lambda_2-\lambda_1) t)}{\int_W u_1 \, d\mu \int u_1 \,d\mu_0}\sum_{k \ge 2} \left|\int_W u_k \, d\mu\right| \left|\int u_k \,d\mu_0\right|\nonumber\\
& \le \frac{\exp(-(\lambda_2-\lambda_1) t)}{\int_W u_1 \, d\mu \int_W u_1 \,d\mu_0}\sqrt{\mu(W)} \sqrt{\int_W \left(\frac{\,d\mu_0}{\,d\mu}\right)^2 d \mu}
\end{align}
so that this quantity is smaller than $1/2$ when $t \ge \frac{C}{\lambda_2-\lambda_1}$, where $C$ is a sufficiently
large constant independent of $f$. Respectively inserting the
inequalities~\eqref{eq:beta1}--\eqref{eq:beta2} and the inequality~\eqref{eq:gamma} at the
numerator and the denominator of~\eqref{eq:alpha}, we obtain that for $t \ge \frac{C}{\lambda_2-\lambda_1}$,
$$e(t) \le  4 \exp(-(\lambda_2 - \lambda_1)t) \sqrt{\mu(W)}  \sqrt{\int_W \left(\frac{\,d\mu_0}{\,d\mu}\right)^2 d \mu} \, \| f\|_{L^\infty} $$
which concludes the proof of
Proposition~\ref{prop:error}.
\end{proof}

\medskip

Note that the
assumption~\eqref{eq:hyp_mu0} on the initial condition $\mu_0$ is not
restrictive. For the conditioned diffusion process, the time
evolution of the density is regularizing. Therefore, if~\eqref{eq:hyp_mu0}
is not satisfied at initial time, the density after a positive time $t_0 
>0$ does satisfy the condition, and we may argue with that density
instead of the initial density in the proof of Proposition~\ref{prop:error}.

\medskip

Proposition~\ref{prop:error} provides an error bound, in total variation norm, on the joint distribution of the exit time from $W$ and the first hitting point on $\partial W$: for $t \ge \frac{C}{\lambda_2-\lambda_1}$,
$$\sup_{f, \|f\|_{L^\infty} \le 1} \Big| \E(f({T_W}-t,X_{T_W}) | {T_W} \ge t) - \E^\nu(f({T_W},X_{T_W})) \Big| \le C \exp(-(\lambda_2-\lambda_1)t).$$
This proposition shows that the correlation time $\tau_{corr}$ should be chosen such that
$$\tau_{corr} \ge \frac{\bar{C}}{\lambda_2-\lambda_1},$$
where $\bar{C}$ is such that $C \exp(-(\lambda_2-\lambda_1)t)$ is small, so that the dephasing step and parallel step, which involve replicas initially   distributed according to the QSD $\nu$ do not introduce a large error in terms of the joint distribution of the exit time from the current state and the next visited state. Notice that one gets a conservative lower bound by taking $\lambda_1=0$, and that $\lambda_2$ may be approximated in practice using an harmonic assumption (namely if $V$ is close to a quadratic function in the well $W$). Within such an approximation, the analysis is also relevant for the Langevin dynamics~\eqref{eq:lang}.

Note that $\displaystyle\frac{1}{\lambda_1}$ is the mean time to leave the well $W$, if the process starts from the QSD. More generally, the mean time to leave the well $W$ is given by
\begin{align*}
\E({T_W})&=\int^\infty \P({T_W} \ge t)\, dt\\
&=\int_0^\infty \int \bar v(t,x) \,d\mu_0 \, dt\\
&=\int_0^\infty \sum_{k \ge 1} \exp(-\lambda_k t) \int_W u_k \, d\mu \int_W u_k \,d\mu_0\, dt\\
&=\sum_{k \ge 1} \frac{1}{\lambda_k} \int_W u_k \, d\mu \int_W u_k \,d\mu_0.
\end{align*}
In order for the algorithm to be efficient, we therefore typically need
that 
\begin{equation}
  \label{eq:adjust}
  \frac{1}{\lambda_2-\lambda_1} \le \tau_{corr} \le \E({T_W}),
\end{equation}
so that, during the decorrelation step, the process reaches the QSD with
good approximation before leaving the well. The pending (and difficult) question is
to make the above estimate more explicit, and therefore practically
useful. Explicitly evaluating $\lambda_2-\lambda_1$ is a question on its
own. Considering more specific situations (metastable well in the limit
of a small parameter, simple 2d periodic examples, ...) could help for
this purpose.

\acknowledgement{The first three authors are greatly and deeply indebted to Arthur Voter for taking the time to patiently explain them the genesis and development of the parallel replica dynamics. This work could not have been possible without his priceless input. CLB and TL acknowledge enlightening 
  discussions with Pablo A. Ferrari (during a workshop in Oberwolfach on large scale stochastic dynamics) and Samuel Herrmann. The second author (TL)
acknowledges support from the Agence Nationale de la Recherche, under grant 
 ANR-09-BLAN-0216-01 (MEGAS). ML and DP acknowledge funding by the US Department of Energy
under award DE-FG02-09ER25880/DE-SC0002085.
Work at Los Alamos National Laboratory (LANL) was funded by the Office
of Science, Office of Advanced Scientific Computing Research.
LANL is operated by Los Alamos National Security, LLC, for the
National Nuclear Security Administration of the U.S. DOE under Contract
No. DE-AC52-06NA25396.
}


\begin{thebibliography}{10}

\bibitem{aldous-flannery-palacios-88}
D.~Aldous, B.~Flannery, and J.L. Palacios.
\newblock Two applications of urn processes: the fringe analysis of search
  trees and the simulation of quasi-stationary distribution of {Markov} chains.
\newblock {\em Prob. in Eng. and Inf. Sciences}, 2:293--307, 1988.

\bibitem{burdzy-holyst-march-00}
K.~Burdzy, R.~Holyst, and P.~March.
\newblock A {Fleming-Viot} particle representation of the {Dirichlet}
  {Laplacian}.
\newblock {\em Communications in Mathematical Physics}, 214(3):679--703, 2000.

\bibitem{cattiaux-collet-lambert-martinez-meleard-san-martin-09}
P.~Cattiaux, P.~Collet, A.~Lambert, S.~Mart{\'{\i}}nez, S.~M{\'e}l{\'e}ard, and
  J.~San~Mart{\'{\i}}n.
\newblock Quasi-stationary distributions and diffusion models in population
  dynamics.
\newblock {\em Ann. Probab.}, 37(5):1926--1969, 2009.

\bibitem{collet-martinez-san-martin-95}
P.~Collet, S.~Mart{\'{\i}}nez, and J.~San~Mart{\'{\i}}n.
\newblock Asymptotic laws for one-dimensional diffusions conditioned to
  nonabsorption.
\newblock {\em Ann. Probab.}, 23(3):1300--1314, 1995.

\bibitem{ferrari-kesten-martinez-picco-95}
P.A. Ferrari, H.~Kesten, S.~Martinez, and P.~Picco.
\newblock Existence of quasi-stationary distributions. a renewal dynamical
  approach.
\newblock {\em Ann. Probab.}, 23(2):511--521, 1995.

\bibitem{ferrari-maric-07}
P.A. Ferrari and N.~Maric.
\newblock Quasi-stationary distributions and {Fleming-Viot} processes in
  countable spaces.
\newblock {\em Electronic Journal of Probability}, 12, 2007.

\bibitem{ferrari-martinez-san-martin-96}
P.A. Ferrari, S.~Martinez, and J.~San~Martin.
\newblock Phase transition for absorbed {B}rownian motion.
\newblock {\em J. Stat. Physics.}, 86(1/2):213--231, 1996.

\bibitem{gilbarg-trudinger-77}
D.~Gilbarg and N.S. Trudinger.
\newblock {\em Elliptic partial differential equations of second order}.
\newblock Springer-Verlag, 1977.

\bibitem{grigorescu-kang-04}
I.~Grigorescu and M.~Kang.
\newblock Hydrodynamic limit for a {Fleming-Viot} type system.
\newblock {\em Stoch. Proc. Appl.}, 110(1):111--143, 2004.

\bibitem{kum-dickson-stuart-uberuaga-voter-04}
O.~Kum, B.M. Dickson, S.J. Stuart, B.P. Uberuaga, and A.F. Voter.
\newblock Parallel replica dynamics with a heterogeneous distribution of
  barriers: Application to n-hexadecane pyrolysis.
\newblock {\em J. Chem. Phys.}, 121:9808, 2004.

\bibitem{lobus-08}
J.U. L\"{o}bus.
\newblock A stationary {Fleming-Viot} type {Brownian} particle system.
\newblock {\em Mathematische Zeitschrift}, 263(3):541--581, 2008.

\bibitem{mandl-61}
P.~Mandl.
\newblock Spectral theory of semi-groups connected with diffusion processes and
  its application.
\newblock {\em Czechoslovak Math. J.}, 11 (86):558--569, 1961.

\bibitem{martinez-san-martin-04}
S.~Mart{\'{\i}}nez and J.~San~Mart{\'{\i}}n.
\newblock Classification of killed one-dimensional diffusions.
\newblock {\em Ann. Probab.}, 32(1A):530--552, 2004.

\bibitem{pinsky-85}
R.G. Pinsky.
\newblock On the convergence of diffusion processes conditioned to remain in a
  bounded region for large time to limiting positive recurrent diffusion
  processes.
\newblock {\em Ann. Probab.}, 13(2):363--378, 1985.

\bibitem{steinsaltz-evans-07}
D.~Steinsaltz and S.N. Evans.
\newblock Quasi-stationary distributions for one-dimensional diffusions with
  killing.
\newblock {\em Trans. Amer. Math. Soc.}, 359(3):1285--1324, 2007.

\bibitem{uberuaga-hoagland-voter-valone-07}
B.P. Uberuaga, R.G. Hoagland, A.F. Voter, and S.M. Valone.
\newblock Direct transformation of vacancy voids to stacking fault tetrahedra.
\newblock {\em Phys. Rev. Lett.}, 99:135501, 2007.

\bibitem{uberuaga-stuart-voter-07}
B.P. Uberuaga, S.J. Stuart, and A.F. Voter.
\newblock Parallel replica dynamics for driven systems: Derivation and
  application to strained nanotubes.
\newblock {\em Phys. Rev. B}, 75:014301, 2007.

\bibitem{voter-98}
A.F. Voter.
\newblock Parallel replica method for dynamics of infrequent events.
\newblock {\em Phys. Rev. B}, 57(22):R13 985, 1998.

\bibitem{voter-montalenti-germann-02}
A.F. Voter, F.~Montalenti, and T.C. Germann.
\newblock Extending the time scale in atomistic simulation of materials.
\newblock {\em Ann. Rev. Mater. Res.}, 32:321--346, 2002.

\end{thebibliography}

\end{document}